\documentclass{siamart250211}
\newtheorem{assumption}{Assumption}

\usepackage{lipsum}
\usepackage{amsfonts}
\usepackage{graphicx}
\usepackage{epstopdf}
\usepackage{algorithmic}
\ifpdf
  \DeclareGraphicsExtensions{.eps,.pdf,.png,.jpg}
\else
  \DeclareGraphicsExtensions{.eps}
\fi

\newsiamremark{remark}{Remark}
\newsiamremark{hypothesis}{Hypothesis}
\crefname{hypothesis}{Hypothesis}{Hypotheses}
\newsiamthm{claim}{Claim}
\newsiamremark{fact}{Fact}
\crefname{fact}{Fact}{Facts}

\headers{Convergence analysis of PPMD}{J. Guo, and D. Xiao}

\title{Convergence analysis of Parametric Probabilistic Manifold Decomposition\thanks{Submitted to the editors DATE.
\funding{ The authors acknowledge the support of the Fundamental Research Funds for the Central Universities, the Top Discipline Plan of Shanghai Universities-Class I, Shanghai Gaofeng Project for University Academic Program Development, National Key R$\&$D Program of China(NO. 2022YFE0208000, 2024YFC2816400, and 2024YFC2816401). This work is also supported in part by grants from the Shanghai Engineering Research Center for Blockchain Applications And Services (No. 19DZ2255100) and the
Shanghai Institute of Intelligent Science and Technology, Tongji University.}}}

\author{Jiaming Guo\thanks{Shanghai Research Institute for Intelligent Autonomous Systems, Tongji University, Shanghai 201210, CHINA. 
 (\email{2411955@tongji.edu.cn}).}
\and Dunhui Xiao\thanks{School of Mathematical Sciences,
Key Laboratory of Intelligent Computing and Applications (Ministry of Education), Tongji University, Shanghai 200092, CHINA. 
  (\email{xiaodunhui@tongji.edu.cn}).}
}

\usepackage{amsopn}

\usepackage{mathrsfs}
\usepackage{amssymb}

\ifpdf
\hypersetup{
  pdftitle={Convergence analysis of Parametric Probabilistic Manifold Decomposition},
  pdfauthor={Jiaming Guo, and Dunhui Xiao}
}
\fi

\begin{document}

\maketitle

\begin{abstract}
This paper presents a convergence analysis for a newly developed nonlinear model reduction method: parametric probabilistic manifold decomposition (PPMD)~\cite{guo2026parametric}. 
In addition, existing analyzes of nonlinear reduced order models typically treat
subspace reduction, manifold representation, regression, and nonlinear
reconstruction as separate components and often remain at the level of
discrete state vectors. To the best of our knowledge, no theory tracks the
complete error propagation in a data-dependent model whose basis, residual
geometry, spectral coordinates, parameter maps, and lifting operator are all
learned from the same numerical solution data. 

We develop a coupled perturbation analysis for the entire PPMD procedure.
A trajectory geometry induced by the spatial discretization and temporal
quadrature connects discrete trajectory vectors isometrically with the
corresponding PDE norm. Population spectral objects are introduced to align
the empirical residual coordinates and derive a uniform coordinate error
estimate, whose propagation through the Hilbert-valued kernel lifting
estimator is then quantified. Combining these results with the full order
discretization error, weighted low-rank approximation, parameter regression,
and residual representation defect yields deterministic and high-probability
trajectory error bounds and consistency in probability in the continuous PDE
trajectory space. The theory identifies how the principal errors interact
and which components limit the accuracy of the nonlinear reduced order model.
\end{abstract}



\section{Introduction}
\label{sec:introduction}

Parametrized incompressible Navier-Stokes problems arise in computational settings that require repeated solution evaluations throughout a prescribed parameter domain. Typical examples include design optimization, uncertainty quantification, parameter identification, and control. Although stable finite element and time discretization methods for incompressible flow are well developed under appropriate regularity assumptions \cite{heywood1982finite,heywood1988finite,heywood1990finite,
liu2007convergence}, repeatedly solving the full discretized model remains computationally expensive. For every new parameter value, one must solve a large discrete system over the entire time interval of interest. The resulting cost can become prohibitive when hundreds or thousands of evaluations are required. This motivates the use of reduced order models that approximate the dependence of the complete solution trajectory on the physical or model parameters.

Projection based reduced order models provide a standard approach to these computationally intensive problems\cite{benner2015survey}. Proper orthogonal decomposition combined with Galerkin projection is a popular method to construct a linear approximation space from snapshots of the full discretized model \cite{gunzburger2017ensemble,luo2009mixed,chaturantabut2012state,
rubino2020numerical,garcia2023pod}. The method is effective when the
parametric solution set is well represented by a linear space of modest
dimension. Their accuracy may get worse when parameter variations shift or deform the main flow structures, in which case a larger number of basis vectors may be needed, and the computational advantage of the reduction is diminished\cite{peherstorfer2020model,rim2023manifold}.

Nonlinear manifold methods instead introduce coordinates adapted to the 
geometry of the snapshot set. Representative techniques include geodesic
distance approximation on neighborhood graphs, kernel embeddings, and
spectral constructions
\cite{tenenbaum2000global,coifman2006diffusion,
wormell2021spectral,dunson2021spectral,calder2022lipschitz,
arias2019unconstrained}. When coupled with regression and reconstruction, these techniques can describe solution sets that are poorly approximated by a single linear space\cite{peng2014online,tencer2021tailored}. Their analysis in the context of a discretized partial differential equation is more delicate because the residual geometry, regression targets, and reconstruction errors should be measured consistently with the spatial and temporal discretization. In addition, the approximation space, graph operator, spectral coordinates, and regression maps are usually
constructed from the same training samples and are therefore statistically dependent.

This work develops an error analysis and consistency theory for parametric
probabilistic manifold decomposition (PPMD)
\cite{guo2026parametric}, a nonlinear reduced order model for complete
parameterized solution trajectories over a fixed time interval. Building on
probabilistic manifold decomposition
\cite{guo2026nonlinear}, PPMD combines a weighted low-rank approximation of
the dominant trajectory component with a nonlinear residual representation
based on spectral coordinates and kernel lifting. The central question is
whether this multistage, data-dependent construction yields a convergent
approximation of the continuous parametric trajectory map, rather than an
autonomous reduced system for time advancement.

This question cannot be answered by applying separate estimates to the
individual components. The weighted basis determines the residual samples,
which determine the graph operator and its spectral coordinates; these
coordinates are then used both as regression targets and as inputs to the
lifting estimator. Thus, the approximation space, residual geometry,
spectral coordinates, parameter maps, and lifting operator are learned from
the same data and are statistically and geometrically dependent. Moreover,
the training trajectories are themselves obtained from a spatial and temporal
discretization of the underlying flow problem. Existing analyses of isolated
manifold, regression, or reconstruction components therefore do not directly
yield a trajectory-level convergence result for PPMD.

We first formulate PPMD in a geometry compatible with the underlying partial
differential equation. For each parameter, the recorded discrete states are
assembled into a trajectory vector whose inner product incorporates the
spatial discretization matrix and temporal quadrature weights. The associated
reconstruction operator preserves this norm exactly. This allows the reduced
model error to be combined directly with the spatial and temporal
discretization error of the full order solver
\cite{choi2019space,tenderini2024space}, rather than being measured only in
the Euclidean norm of concatenated coefficient vectors.

The main analytical difficulty arises from the nonlinear residual
representation. PPMD decomposes each normalized trajectory into a weighted
low-rank component and a residual correction described by spectral
coordinates of a normalized graph operator
\cite{tenenbaum2000global,coifman2006diffusion,
wormell2021spectral,dunson2021spectral,calder2022lipschitz,
arias2019unconstrained}. We introduce population comparison objects, match
the empirical and population modes, resolve their sign ambiguity, and derive
a uniform estimate for the learned residual coordinates. The analysis
separates graph-operator approximation from uniform eigenfunction
approximation because the spectral coordinates enter the subsequent
regression and reconstruction stages pointwise.

Residual trajectories are reconstructed by a Hilbert-valued kernel ridge
regression map
\cite{micchelli2005learning,caponnetto2007optimal,
li2024towards,celisse2021analyzing}. Since both its inputs and Gram matrix are
formed from empirical spectral coordinates, we quantify how coordinate
perturbations propagate through the lifting estimator. The dependence of the
linear and nonlinear coordinates on the parameter is treated either by
kernel ridge regression with vector outputs or, for ordered one-dimensional
samples, by parameter continuation followed by weighted smoothing splines
\cite{cox1984multivariate,ragozin1983error,
wahba1990spline,de2001calculation}. The corresponding estimates are stated
conditionally on the learned representation.

The resulting coupled analysis accounts for the full order discretization
error, weighted low-rank approximation, residual spectral-coordinate error,
lifting perturbation, parameter-regression error, and residual
representation defect. Their combination yields deterministic and
high-probability trajectory reconstruction bounds. Under explicit assumptions
on the flow discretization, residual geometry, spectral approximation,
kernel regression, spline stability, and representation accuracy, we further
establish consistency in probability in the continuous PDE trajectory space.
The analysis therefore explains how errors generated by the mutually
dependent components of PPMD interact in the final approximation and provides
a theoretical basis for the componentwise error diagnosis in the numerical
experiments.

The remainder of the paper is organized as follows.
Section~\ref{sec:ppmd_method} presents PPMD in the weighted trajectory
setting. Section~\ref{sec:ppmd_error_analysis} develops the component
estimates, reconstruction bounds, and consistency result. The subsequent
sections present the numerical experiments, followed by the conclusions and
limitations of the analysis.

\section{Parametric probabilistic manifold decomposition}
\label{sec:ppmd_method}

PPMD combines a weighted low-rank representation with nonlinear coordinates for the unresolved residual and a learned lifting from these coordinates to complete solution trajectories.

\subsection{Continuous and discrete trajectory spaces}
\label{subsec:trajectory_spaces}

Let \(\Omega\subset\mathbb R^d\), \(d=2\) or \(3\), be a bounded Lipschitz domain, let \(T>0\), and let \(\mathcal D\subset\mathbb R^p\) be compact.
\[
H:=\overline{\{\boldsymbol v\in C_0^\infty(\Omega)^d:
                 \nabla\!\cdot\boldsymbol v=0\}}^{\,L^2(\Omega)^d},
\quad
V:=\overline{\{\boldsymbol v\in C_0^\infty(\Omega)^d:
                 \nabla\!\cdot\boldsymbol v=0\}}^{\,H_0^1(\Omega)^d}.
\]
For a fixed \(\chi\ge0\), set
\begin{equation}
X_\chi:=
\begin{cases}
H, & \chi=0,\\
V, & \chi>0,
\end{cases}
\qquad
\mathcal X_T:=L^2(0,T;X_\chi),
\end{equation}
and equip \(\mathcal X_T\) with
\begin{equation}
\langle v,w\rangle_{\mathcal X_T}:=
\begin{cases}
\displaystyle\int_0^T(v,w)_{L^2(\Omega)}\,dt,
   & \chi=0,\\[1mm]
\displaystyle\int_0^T\!\left[
 (v,w)_{L^2(\Omega)} +
 \chi(\nabla v,\nabla w)_{L^2(\Omega)}
\right]dt,
   & \chi>0.
\end{cases}
\label{eq:continuous_trajectory_inner_product}
\end{equation}
This piecewise definition avoids requiring \(L^2\) gradients when
\(\chi=0\).  The finite-time solution map is
\begin{equation}
\mathcal S:\mathcal D\to\mathcal X_T,
\qquad
\mathcal S(\mu):=u(\cdot;\mu)|_{[0,T]}.
\end{equation}
We assume the parameter-independent initial value
\(u(0;\mu)=u_0\); it is omitted from the stacked data because it carries no
parametric variability and does not affect the \(L^2\)-in-time trajectory
norm.

Let \(V_h\subset V\) have basis
\(\{\zeta_\alpha\}_{\alpha=1}^{N_h}\), and let
\(0=t_0<t_1<\cdots<t_m=T\), with
\(\tau_n=t_n-t_{n-1}>0\) and \(\tau:=\max_{1\le n\le m}\tau_n\). Denote by \(u_h^n(\mu)\in V_h\) the fully discrete solution at \(t_n\), and by
\(\mathbf u^n(\mu)\in\mathbb R^{N_h}\) its coefficient vector in this
basis.  For \(N=mN_h\), define the stacked trajectory
\begin{equation}
\mathbf u_{h,\tau}(\mu):=
\operatorname{col}\{\mathbf u^1(\mu),\ldots,\mathbf u^m(\mu)\}
\in\mathbb R^N.
\end{equation}
The mass and stiffness matrices and their trajectory metric are
\begin{equation}
(M_h)_{\alpha\beta}:=(\zeta_\alpha,\zeta_\beta)_{L^2(\Omega)},
\quad
(A_h)_{\alpha\beta}:=(\nabla\zeta_\alpha,\nabla\zeta_\beta)_{L^2(\Omega)},
\quad
G_h:=M_h+\chi A_h.
\end{equation}
Because \(M_h\) is symmetric positive definite, \(A_h\) is positive
semidefinite, and \(\tau_n>0\), both \(G_h\) and
\begin{equation}
W_{h,\tau}:=\operatorname{diag}(\tau_1G_h,\ldots,\tau_mG_h)
\end{equation}
are symmetric positive definite.  Hence
\(W_{h,\tau}^{1/2}\) and \(W_{h,\tau}^{-1/2}\) are well defined.  We write
\(\langle a,b\rangle_{h,\tau}:=a^\top W_{h,\tau}b\) and
\(\|a\|_{h,\tau}:=\langle a,a\rangle_{h,\tau}^{1/2}\). 

To connect coefficient vectors with trajectories, define
\begin{equation}
\mathcal E_{h,\tau}:\mathbb R^N\to L^2(0,T;V_h)\subset\mathcal X_T,\quad (\mathcal E_{h,\tau}\mathbf v)(t):=v_h^n,
\quad t\in(t_{n-1},t_n], \quad n=1,\ldots,m.\label{eq:piecewise_constant_time_reconstruction}
\end{equation}
where the \(n\)th block of \(\mathbf v\) is the coefficient vector of
\(v_h^n\).  Consequently,
\begin{equation}
\langle\mathcal E_{h,\tau}a,\mathcal E_{h,\tau}b\rangle_{\mathcal X_T}
=\langle a,b\rangle_{h,\tau}.
\end{equation}

\subsection{Normalization and weighted low-rank representation}
\label{subsec:weighted_low_rank}

Let \(\{\mu_i\}_{i=1}^{n_s}\subset\mathcal D\) be the training parameters. The normalization is the invertible affine map
\begin{equation}
\mathcal N(v):=B(v-\mathbf m_{\mathcal N}),
\qquad
\mathcal N^{-1}(a):=B^{-1}a+\mathbf m_{\mathcal N},
\end{equation}
where \(B\in\mathbb R^{N\times N}\) is nonsingular. If
\(\mathbf m_{\mathcal N}\) or \(B\) is estimated from the training data,
all subsequent probability statements are understood conditionally on
these quantities. Along a discretization sequence, we assume
\begin{equation}
\|B^{-1}\|_{\mathcal L(\mathbb R^N,\|\cdot\|_{h,\tau})}
\le L_{\mathcal N^{-1}}<\infty
\end{equation}
with a uniform constant. Define
$\bar{\mathbf u}_{h,\tau}(\mu)
:=
\mathcal N\bigl(\mathbf u_{h,\tau}(\mu)\bigr)$, $\bar{\mathbf u}_i:=\mathcal N(\mathbf u_{h,\tau}(\mu_i))$ and
\(\bar U:=[\bar{\mathbf u}_1,\ldots,\bar{\mathbf u}_{n_s}]\).

The weighted snapshot matrix and its singular value decomposition are
\begin{equation}
\widetilde U:=W_{h,\tau}^{1/2}\bar U
             =\widetilde\Psi\Sigma\Theta^\top,
\qquad
\widetilde{\mathbf u}_i:=W_{h,\tau}^{1/2}\bar{\mathbf u}_i.
\label{eq:weighted_snapshot_matrix}
\end{equation}
Thus
\(\widetilde{\mathbf u}_i^\top\widetilde{\mathbf u}_j
=\langle\bar{\mathbf u}_i,\bar{\mathbf u}_j\rangle_{h,\tau}\).
Let \(\widetilde\Psi_r\) contain the first \(r\) columns of
\(\widetilde\Psi\), corresponding to the \(r\) largest singular values, and define
\begin{equation}
\Psi_r:=W_{h,\tau}^{-1/2}\widetilde\Psi_r,
\qquad
\Psi_r^\top W_{h,\tau}\Psi_r=I_r.
\end{equation}
The linear coordinates, weighted projector, and residual targets are
\begin{align}
z_i:=\Psi_r^\top W_{h,\tau}&\bar{\mathbf u}_i,
\qquad
P_r^{h,\tau}:=\Psi_r\Psi_r^\top W_{h,\tau},
\qquad
\mathbf r_i:=\bar{\mathbf u}_i-\Psi_rz_i, \label{eq:resdiual}\\
Z&:=[z_1,\ldots,z_{n_s}],
\qquad
R:=[\mathbf r_1,\ldots,\mathbf r_{n_s}],
\end{align}
Accordingly,
\(\bar{\mathbf u}_i=\Psi_rz_i+\mathbf r_i\). The unresolved component is retained as the residual target for the nonlinear correction.

\subsection{Residual manifold and spectral coordinates}
\label{subsec:residual_coordinates}

Let \(\mathcal R_{n_s}:=\{\mathbf r_i\}_{i=1}^{n_s}\). We construct a
connected \(k\)-nearest-neighbor graph using the weighted edge lengths
\(\|\mathbf r_i-\mathbf r_j\|_{h,\tau}\), and we denote its shortest-path
distance by \(d_G(\mathbf r_i,\mathbf r_j)\). If the initial graph is
disconnected, \(k\) is increased until connectivity is obtained; no samples
are discarded.

For \(\varepsilon_{\rm res}>0\), define
\begin{equation}
A_{ij}
:=
c_{\varepsilon_{\rm res}}
\exp\!\left[
-\frac{d_G(\mathbf r_i,\mathbf r_j)^2}
{\varepsilon_{\rm res}^2}
\right],
\qquad
d_i:=\sum_jA_{ij},
\qquad
\pi_i:=\frac{d_i}{\sum_jd_j},
\end{equation}
where the positive factor \(c_{\varepsilon_{\rm res}}\) cancels under
normalization. Set
\begin{equation}
D_A:=\operatorname{diag}(d_1,\ldots,d_{n_s}),\qquad
P_{\rm res}:=D_A^{-1}A,
\qquad
S_{\rm res}:=D_A^{-1/2}AD_A^{-1/2}.
\end{equation}
The two matrices are similar. Let
\begin{equation}
S_{\rm res}v_\ell=\vartheta_\ell v_\ell,
\qquad
v_\ell^\top v_k=\delta_{\ell k},
\qquad
1=\vartheta_0>\vartheta_1\ge\vartheta_2\ge\cdots,
\end{equation}
and define the stationary-measure normalized Markov eigenvectors by
\begin{equation}
\psi_\ell
:=
\left(\sum_jd_j\right)^{1/2}D_A^{-1/2}v_\ell,
\qquad
P_{\rm res}\psi_\ell=\vartheta_\ell\psi_\ell,
\qquad
\sum_i\pi_i\psi_\ell(i)\psi_k(i)=\delta_{\ell k}.\label{eq:markov}
\end{equation}
Connectivity and positivity of the affinity imply that the trivial pair
\((\vartheta_0,\psi_0)=(1,\mathbf1)\) is simple; it is excluded from the
coordinates. For \(t_{\rm res}\in\mathbb N\), \(t_{\rm res}\ge1\), define
\begin{equation}
\Phi:=[\phi_1,\ldots,\phi_{n_s}],
\qquad
\phi_i:=
\bigl(
\vartheta_1^{t_{\rm res}}\psi_1(i),\ldots,
\vartheta_q^{t_{\rm res}}\psi_q(i)
\bigr)^\top
\in\mathbb R^q.
\label{eq:empirical_residual_coordinates}
\end{equation}
The algorithm selects these empirical eigenpairs; only in the analysis are they matched with the isolated population spectral neighborhoods introduced in the residual spectral approximation assumption.

\subsection{Parametric coordinate maps}
\label{subsec:parameter_regression}

\subsubsection{Direct parameter regression}
\label{subsubsec:direct_parameter_regression}

For any \(p\ge1\), both latent maps can be learned directly by
vector-valued kernel ridge regression (KRR). Let
\(\mathcal H_{\mu,z}\) and \(\mathcal H_{\mu,\phi}\) be vector-valued Reproducing Kernel Hilbert Spaces (RKHSs) over \(\mathcal D\). We define
\begin{align}
\widehat z &:=
\arg\min_{f\in\mathcal H_{\mu,z}}
\left\{\frac1{n_s}\sum_{i=1}^{n_s}\|z_i-f(\mu_i)\|_2^2
+\lambda_{\mu,z}\|f\|_{\mathcal H_{\mu,z}}^2\right\},\\
\widehat\phi_{\rm p}&:=
\arg\min_{g\in\mathcal H_{\mu,\phi}}
\left\{\frac1{n_s}\sum_{i=1}^{n_s}\|\phi_i-g(\mu_i)\|_2^2
+\lambda_{\mu,\phi}\|g\|_{\mathcal H_{\mu,\phi}}^2\right\}.
\end{align}
The kernels and regularization parameters are selected using only the
training data. Each latent map is fitted in an RKHS with vector outputs. With the separable kernels used here, all coordinate components in a latent block share the same Gram matrix.

\subsubsection{Ordered one-dimensional continuation}
\label{subsubsec:one_dimensional_continuation}

For \(\mathcal D=[a,b]\) and ordered parameters
\(\mu_1<\cdots<\mu_{n_s}\), continuation followed by smoothing splines
provides an alternative to direct KRR. We assume a fixed continuation step;
for a nonuniform grid, the current parameter and step size must be included
as inputs to the continuation map. Let
\begin{equation}
d_z:=r,\qquad d_\phi:=q,\qquad
y_i^z:=z_i,\qquad y_i^\phi:=\phi_i.
\end{equation}
For \(\ell\in\{z,\phi\}\), define
\begin{equation}
\widehat T_\ell:=
\arg\min_{T\in\mathcal H_\ell}
\left\{
\frac1{n_s-1}\sum_{i=1}^{n_s-1}
\|y_{i+1}^\ell-T(y_i^\ell)\|_2^2
+
\lambda_\ell^{\rm cont}\|T\|_{\mathcal H_\ell}^2
\right\}.
\label{eq:latent_continuation_compact}
\end{equation}
Set \(\widetilde y_i^\ell:=y_i^\ell\) for \(i\le n_s\), and generate
$\widetilde y_{i+1}^\ell := \widehat T_\ell(\widetilde y_i^\ell)$ for $i=n_s,\ldots,N_p-1$. With
\begin{equation}
\omega_i^\ell:=
\begin{cases}
1, & i\le n_s,\\
\gamma_\ell, & i>n_s,
\end{cases}
\qquad
N_{\rm e}^\ell:=\sum_{i=1}^{N_p}\omega_i^\ell,\label{eq:weight}
\end{equation}
define
\begin{equation}
\widehat y_{\rm s}^{\,\ell}
:=
\arg\min_{s\in H^2(a,b;\mathbb R^{d_\ell})}
\left\{
\frac1{N_{\rm e}^\ell}
\sum_{i=1}^{N_p}
\omega_i^\ell
\|\widetilde y_i^\ell-s(\mu_i)\|_2^2
+
\alpha_\ell\int_a^b\|s''(\mu)\|_2^2\,d\mu
\right\}.
\label{eq:weighted_vector_spline_compact}
\end{equation}
Finally, set $\widehat z:=\widehat y_{\rm s}^{\,z}$, and $\widehat\phi_{\rm p}:=\widehat y_{\rm s}^{\,\phi}.$ We assume
\(0<\gamma_{\min}\le\gamma_\ell\le1\), where
\(\gamma_{\min}\) is independent of \(n_s\) along the approximation sequence.

\subsection{Residual lifting and reconstruction}
\label{subsec:residual_lifting}

Let \(\mathcal H_{\mathcal K}\) be a Hilbert-valued RKHS of maps
\(\mathbb R^q\to \mathbb R^N_{h,\tau}\) with $\mathbb R^N_{h,\tau}:=
(\mathbb R^N,\langle\cdot,\cdot\rangle_{h,\tau}).$ The empirical residual-lifting estimator is
\begin{equation}
\widehat{\mathcal K}_{\rm p}:=
\arg\min_{F\in\mathcal H_{\mathcal K}}
\left\{\frac1{n_s}\sum_{i=1}^{n_s}
\|\mathbf r_i-F(\phi_i)\|_{h,\tau}^2
+\lambda_{\mathcal K}\|F\|_{\mathcal H_{\mathcal K}}^2\right\}.
\label{eq:residual_lifting_krr}
\end{equation}
A convenient scalar input kernel is, for example,
\begin{equation}
k_{\mathcal K}(\phi,\phi')
=\bigl(\gamma_{\mathcal K}\phi^\top\phi'+c_{\mathcal K}\bigr)^{d_{\mathcal K}},
\qquad
\gamma_{\mathcal K}>0,\quad c_{\mathcal K}\ge0,\quad
d_{\mathcal K}\in\mathbb N.
\end{equation}

For either construction of the parametric coordinate maps, the normalized
trajectory reconstructed by PPMD is
\begin{equation}
\widehat{\bar{\mathbf u}}_{r,h,\tau}(\mu)
:=
\Psi_r\widehat z(\mu)
+
\widehat{\mathcal K}_{\rm p}
\bigl(\widehat\phi_{\rm p}(\mu)\bigr).
\label{eq:ppmd_normalized_reconstruction}
\end{equation}
The coefficient trajectory in the original variables and the corresponding
finite-time reduced order model is
\begin{equation}
\widehat{\mathbf u}_{r,h,\tau}(\mu)
:=
\mathcal N^{-1}
\bigl(\widehat{\bar{\mathbf u}}_{r,h,\tau}(\mu)\bigr),
\qquad
\widehat{\mathcal S}_{r,h,\tau}(\mu)
:=
\mathcal E_{h,\tau}
\widehat{\mathbf u}_{r,h,\tau}(\mu)
\in\mathcal X_T.
\label{eq:ppmd_final_reconstruction}
\end{equation}
Thus, all low-rank projections, residual distances, lifting targets, and
reconstruction errors use the PDE-induced trajectory geometry, while the
residual coordinates are constructed from the associated weighted graph.

\section{Error analysis of PPMD}
\label{sec:ppmd_error_analysis}

This section uses all algorithmic objects defined in
Section~\ref{sec:ppmd_method} without redefining them.  In particular,
\(\mathcal X_T\), \(W_{h,\tau}\), \(\mathcal E_{h,\tau}\),
\(\mathcal N\), \(\Psi_r\), \(P_r^{h,\tau}\), the empirical residual
coordinates \(\phi_i\), the two parameter-regression constructions,
\(\widehat{\mathcal K}_{\rm p}\), and
\(\widehat{\mathcal S}_{r,h,\tau}\) retain the meanings assigned there.

The purpose of the present section is only to introduce the population
comparison objects and the assumptions needed to quantify their errors.
All random quantities are defined on the joint training experiment; in
particular, the bounds below account for the fact that the normalization, weighted basis, residual graph, and sign alignment are learned from the same training set.  In accordance with the convention in Section~\ref{subsec:weighted_low_rank}, when the normalization is estimated from data, each stated probability bound is assumed under the corresponding regular conditional law given the learned normalization quantities.

\subsection{Analysis targets and discretization error}
\label{subsec:error_analysis_targets}

We assume that the fully discrete solver satisfies the following estimate uniformly over the parameter domain:
\begin{equation}
\sup_{\mu\in\mathcal D}
\left\|
\mathcal S(\mu)
-\mathcal E_{h,\tau}\mathbf u_{h,\tau}(\mu)
\right\|_{\mathcal X_T}
\le
E_{\rm d}(h,\tau)
:=
C_{\rm disc}^{\ast}\bigl(h^{s_h}+\tau^{s_\tau}\bigr).
\label{eq:uniform_discretization_error}
\end{equation}
This is an assumption on the selected spatial and temporal discretizations\cite{heywood1982finite,heywood1988finite,shen1992error,
heywood1990finite}, not an automatic consequence of the PPMD construction.

For an arbitrary \(\mu\in\mathcal D\), extend only the training coordinate definitions from Section~\ref{subsec:weighted_low_rank} by setting
\begin{equation}
z_\star(\mu)
:=\Psi_r^\top W_{h,\tau}\bar{\mathbf u}_{h,\tau}(\mu),
\qquad
\mathbf r_\star(\mu)
:=(I-P_r^{h,\tau})\bar{\mathbf u}_{h,\tau}(\mu).
\label{eq:analysis_linear_and_residual_targets}
\end{equation}
The weighted orthonormality of \(\Psi_r\) gives the exact decomposition
\begin{equation}
\bar{\mathbf u}_{h,\tau}(\mu)
=\Psi_r z_\star(\mu)+\mathbf r_\star(\mu),
\qquad
\|\Psi_r a\|_{h,\tau}=\|a\|_2.
\label{eq:exact_weighted_decomposition}
\end{equation}
At a training parameter \(\mu_i\), these quantities coincide with \(z_i\)
and \(\mathbf r_i\) from the equation \eqref{eq:resdiual}.

\subsection{Population residual geometry and spectral approximation}
\label{subsec:residual_spectral_error}

Assume that the residuals lie in a tubular neighborhood
\(\mathcal U_{\mathcal M}\) of a compact
\(d_{\mathcal M}\)-dimensional manifold
\(\mathcal M_r^{h,\tau}\), on which the nearest-point projection
\[
\Pi_{\mathcal M}:\mathcal U_{\mathcal M}\to\mathcal M_r^{h,\tau}
\]
is uniquely defined. Write \(x_i:=\Pi_{\mathcal M}(\mathbf r_i)\).
Let \(dV\) be the Riemannian volume measure on
\(\mathcal M_r^{h,\tau}\), and assume that the \(x_i\) have density
\(\rho_{\mathcal M}\) with respect to \(dV\), bounded above and below by
positive constants. Let
\(\operatorname{dist}_{\mathcal M}\) denote the intrinsic distance of \(\mathcal M_r^{h,\tau}\). For \(\varepsilon=\varepsilon_{\rm res}\), define
\begin{align}
K_\varepsilon(x,y)
:=&
c_\varepsilon
\exp\!\left(
-\frac{\operatorname{dist}_{\mathcal M}(x,y)^2}{\varepsilon^2}
\right),
\qquad c_\varepsilon>0,\\
d_\varepsilon(x)
:=
\int_{\mathcal M_r^{h,\tau}}
K_\varepsilon(x,y)&\rho_{\mathcal M}(y)\,dV(y),
\qquad
Z_\varepsilon
:=
\int_{\mathcal M_r^{h,\tau}}
d_\varepsilon(x)\rho_{\mathcal M}(x)\,dV(x).
\label{eq:population_degree_and_normalizer}
\end{align}
We assume
\begin{equation}
0<d_{\min}(\varepsilon)
\le d_\varepsilon(x)
\le d_{\max}(\varepsilon)<\infty
\qquad
\text{for all }x\in\mathcal M_r^{h,\tau}.
\label{eq:population_degree_bounds}
\end{equation}
The population Markov operator and its symmetric conjugate are
\begin{align}
(\mathcal P_\varepsilon f)(x)
&:=
\frac{1}{d_\varepsilon(x)}
\int_{\mathcal M_r^{h,\tau}}
K_\varepsilon(x,y)f(y)\rho_{\mathcal M}(y)\,dV(y),
\label{eq:population_markov_operator}\\
(\mathcal S_\varepsilon f)(x)
&:=
\int_{\mathcal M_r^{h,\tau}}
\frac{K_\varepsilon(x,y)}
{\sqrt{d_\varepsilon(x)d_\varepsilon(y)}}
f(y)\rho_{\mathcal M}(y)\,dV(y).
\label{eq:population_symmetric_operator}
\end{align}
They are unitarily equivalent after changing from
\(\rho_{\mathcal M}dV\) to the stationary probability measure
\[
d\pi_\varepsilon(x)
:=
Z_\varepsilon^{-1}d_\varepsilon(x)\rho_{\mathcal M}(x)\,dV(x).
\]
Let
\((\vartheta_\ell^\star,v_\ell^\star)\), \(1\le\ell\le q\), be selected nontrivial simple and isolated eigenpairs of \(\mathcal S_\varepsilon\), with \(\|v_\ell^\star\|_{L^2_{\rho_{\mathcal M}}}=1\), and define the Markov eigenfunctions, normalized with respect to the stationary measure, by
\begin{equation}
\psi_\ell^\star(x)
:=
\frac{\sqrt{Z_\varepsilon}}{\sqrt{d_\varepsilon(x)}}v_\ell^\star(x).
\label{eq:population_markov_eigenfunction}
\end{equation}
Then
\(\mathcal P_\varepsilon\psi_\ell^\star
=\vartheta_\ell^\star\psi_\ell^\star\) and
\(\|\psi_\ell^\star\|_{L^2_{\pi_\varepsilon}}=1\). This normalization is the population counterpart of the empirical normalization in
Section~\ref{subsec:residual_coordinates}.

For comparison with \(S_{\rm res}\), introduce the sample matrix normalized with respect to the population degree function
\begin{equation}
\left(S_{\varepsilon,n_s}^{\circ}\right)_{ij}
:=
\frac{1}{n_s}
\frac{K_\varepsilon(x_i,x_j)}
{\sqrt{d_\varepsilon(x_i)d_\varepsilon(x_j)}},
\qquad 1\le i,j\le n_s.
\label{eq:population_normalized_sample_matrix}
\end{equation}
Let \(\vartheta_{\ell,n_s}^{\circ}\) be its eigenvalue matched with
\(\vartheta_\ell^\star\). We impose the following spectral approximation
assumption\cite{wormell2021spectral,dunson2021spectral,
calder2022lipschitz,arias2019unconstrained,garcia2020error}.

\begin{assumption}[Residual spectral approximation]
\label{ass:residual_spectral_approximation}
For every \(\delta_{\rm op}\in(0,1)\), there is an event
\(\mathcal A_{\rm op}\), with
\(\mathbb P(\mathcal A_{\rm op}^c)\le\delta_{\rm op}\), on which
\begin{align}
\|S_{\rm res}-S_{\varepsilon,n_s}^{\circ}\|_{\rm op}
&\le \mathfrak E_{\rm g}(\delta_{\rm op}),
\label{eq:residual_operator_approximation}\\
\max_{1\le\ell\le q}
\left|\vartheta_{\ell,n_s}^{\circ}-\vartheta_\ell^\star\right|
&\le \mathfrak E_{\rm s}(\delta_{\rm op}).
\label{eq:sample_population_eigenvalue_error}
\end{align}
The first term includes graph-distance distortion
\cite{arias2019unconstrained}, empirical degree error, sampling error, and deviations of the residuals from the underlying manifold. No additional bandwidth bias is required here because
\eqref{eq:residual_operator_approximation} compares objects constructed
with the same bandwidth. Set
\[
\mathfrak E_\vartheta(\delta_{\rm op})
:=
\mathfrak E_{\rm g}(\delta_{\rm op})
+\mathfrak E_{\rm s}(\delta_{\rm op}), \quad
\gamma_{\rm res}
:=
\min_{1\le\ell\le q}
\operatorname{dist}\!\left(
\vartheta_\ell^\star,
\operatorname{spec}(\mathcal S_\varepsilon)
\setminus\{\vartheta_\ell^\star\}
\right) > 0.
\]
For
\(I_\ell:=(\vartheta_\ell^\star-\gamma_{\rm res}/2,
\vartheta_\ell^\star+\gamma_{\rm res}/2)\), assume that
\(\mathfrak E_\vartheta(\delta_{\rm op}) < \frac{\gamma_{\rm res}}2\), that \(I_\ell\) contains
exactly one eigenvalue of \(S_{\varepsilon,n_s}^{\circ}\), counted with
multiplicity, and that this eigenvalue is separated from the rest of the
sampled spectrum by more than \(2\mathfrak E_{\rm g}(\delta_{\rm op})\). Consequently, \(I_\ell\) contains exactly
one eigenvalue of \(S_{\rm res}\). Under the common ordering of the empirical and population eigenpairs,
\begin{equation}
\left|\vartheta_\ell-\vartheta_\ell^\star\right|
\le \mathfrak E_\vartheta(\delta_{\rm op}),
\qquad 1\le\ell\le q,
\label{eq:matched_eigenvalue_error}
\end{equation}
where \(\vartheta_\ell\) is the empirical eigenvalue used in
\eqref{eq:empirical_residual_coordinates}.
\end{assumption}

The rank-one condition in Assumption~\ref{ass:residual_spectral_approximation}
is needed for unambiguous mode matching; a Hausdorff spectral bound alone does not exclude multiple empirical eigenvalues in the same spectral neighborhood\cite{davis1970rotation}.

\begin{assumption}[Uniform empirical eigenfunction approximation]
\label{ass:uniform_eigenfunction_approximation}
For every \(\delta_{\rm sp}\in(0,1)\), there is an event
\(\mathcal A_{\rm sp}\), with
\(\mathbb P(\mathcal A_{\rm sp}^c)\le\delta_{\rm sp}\), on which signs
\(s_\ell\in\{-1,1\}\) can be selected\cite{calder2022lipschitz,dunson2021spectral} so that the empirical Markov eigenfunctions \(\psi_\ell(i)\) defined in the equation \eqref{eq:markov} satisfy
\begin{equation}
\max_{\substack{1\le i\le n_s\\1\le\ell\le q}}
\left|
s_\ell\psi_\ell(i)-\psi_\ell^\star(x_i)
\right|
\le
\mathfrak E_{\psi,\infty}(\delta_{\rm sp}).
\label{eq:uniform_eigenfunction_error}
\end{equation}
\end{assumption}

Let
\[
J_\Phi:=\operatorname{diag}(s_1,\ldots,s_q),
\qquad
M_\psi:=
\max_{1\le\ell\le q}
\|\psi_\ell^\star\|_{L^\infty(\mathcal M_r^{h,\tau})} < \infty,
\]
and define only the two comparison coordinates
\begin{equation}
\phi_i^{\rm al}:=J_\Phi\phi_i,
\qquad
\phi_i^\star
:=
\left(
(\vartheta_1^\star)^{t_{\rm res}}\psi_1^\star(x_i),
\ldots,
(\vartheta_q^\star)^{t_{\rm res}}\psi_q^\star(x_i)
\right)^\top.
\label{eq:aligned_and_ideal_coordinates}
\end{equation}
The empirical coordinate \(\phi_i\) itself is not redefined here.

\begin{proposition}[Uniform residual-coordinate error]
\label{prop:uniform_residual_coordinate_error}
On \(\mathcal A_{\rm op}\cap\mathcal A_{\rm sp}\),
\begin{equation}
\max_{1\le i\le n_s}
\|\phi_i^{\rm al}-\phi_i^\star\|_2
\le
E_{\rm b}
:=
\sqrt q\left[
\mathfrak E_{\psi,\infty}(\delta_{\rm sp})
+t_{\rm res}M_\psi
\mathfrak E_\vartheta(\delta_{\rm op})
\right].
\label{eq:embedding_error}
\end{equation}
\end{proposition}

\begin{proof}
Fix \(i\in\{1,\ldots,n_s\}\) and
\(\ell\in\{1,\ldots,q\}\). By
\eqref{eq:aligned_and_ideal_coordinates}, the \(\ell\)th component of the
coordinate error can be decomposed as
\[
[\phi_i^{\rm al}-\phi_i^\star]_\ell
={}
\vartheta_\ell^{t_{\rm res}}
\bigl[s_\ell\psi_\ell(i)-\psi_\ell^\star(x_i)\bigr] 
+
\bigl[
\vartheta_\ell^{t_{\rm res}}
-(\vartheta_\ell^\star)^{t_{\rm res}}
\bigr]\psi_\ell^\star(x_i).
\]
Because the empirical and population Markov operators are contractions, $|\vartheta_\ell|\le1$ and $|\vartheta_\ell^\star|\le1$. Moreover, for any \(a,b\in[-1,1]\) and integer \(t\ge1\),
\[
|a^t-b^t|
=
|a-b|
\left|
\sum_{k=0}^{t-1}a^{t-1-k}b^k
\right|
\le t|a-b|.
\]
Hence, on \(\mathcal A_{\rm op}\cap\mathcal A_{\rm sp}\),
\[
|[\phi_i^{\rm al}-\phi_i^\star]_\ell|
\le
|s_\ell\psi_\ell(i)-\psi_\ell^\star(x_i)|
+t_{\rm res}
|\vartheta_\ell-\vartheta_\ell^\star|
|\psi_\ell^\star(x_i)| 
\le
\mathfrak E_{\psi,\infty}(\delta_{\rm sp})
+t_{\rm res}M_\psi
\mathfrak E_\vartheta(\delta_{\rm op}).
\]
This bound is uniform in both \(i\) and \(\ell\). Therefore,
\[
\|\phi_i^{\rm al}-\phi_i^\star\|_2
=
\left(
\sum_{\ell=1}^q
|[\phi_i^{\rm al}-\phi_i^\star]_\ell|^2
\right)^{1/2}\le \sqrt q\left[
\mathfrak E_{\psi,\infty}(\delta_{\rm sp})
+t_{\rm res}M_\psi
\mathfrak E_\vartheta(\delta_{\rm op})
\right]
=E_{\rm b}.
\]
Taking the maximum over \(1\le i\le n_s\) proves
\eqref{eq:embedding_error}.
\end{proof}

\subsection{Residual lifting error}
\label{subsec:residual_lifting_error}

Let \(k_{\mathcal K}\) be the scalar input kernel of the separable
Hilbert-valued RKHS\cite{micchelli2005learning,carmeli2006vector} \(\mathcal H_{\mathcal K}\) used in
\eqref{eq:residual_lifting_krr}. Assume that all ideal coordinates and all aligned learned coordinates evaluated on \(\mathcal D\), including their training values, belong to a compact set
\(\mathcal Z\subset\mathbb R^q\), and that
\begin{align}
\sup_{\phi\in\mathcal Z}k_{\mathcal K}(\phi,\phi)
&\le\kappa_{\mathcal K}^2,
\label{eq:lifting_kernel_bound}\\
\left\|
k_{\mathcal K}(\cdot,\phi)
-k_{\mathcal K}(\cdot,\eta)
\right\|_{\mathcal H_{k_{\mathcal K}}}
&\le L_{\rm f}\|\phi-\eta\|_2,
\qquad \phi,\eta\in\mathcal Z.
\label{eq:lifting_feature_lipschitz}
\end{align}
We also assume sign invariance,
\begin{equation}
k_{\mathcal K}(J\phi,J\eta)
=k_{\mathcal K}(\phi,\eta)
\label{eq:lifting_kernel_sign_invariance}
\end{equation}
for every diagonal sign matrix \(J\) relevant to the selected modes. The polynomial kernel displayed in the method section \eqref{subsec:residual_lifting} satisfies euqation \eqref{eq:lifting_kernel_sign_invariance}; on a compact set, it also satisfies equations \eqref{eq:lifting_kernel_bound}-\eqref{eq:lifting_feature_lipschitz}.

Let \(\nu_{\mathcal Z}\) be the ideal-coordinate probability measure and
define the scalar input covariance operator
\begin{equation}
(T_{\mathcal K}^{\rm in}f)(\cdot)
:=
\int_{\mathcal Z}
k_{\mathcal K}(\cdot,\phi)f(\phi)\,
d\nu_{\mathcal Z}(\phi),
\qquad
T_{\mathcal K}
:=
T_{\mathcal K}^{\rm in}
\otimes I_{\mathbb R^N_{h,\tau}}
:
\mathcal H_{\mathcal K}\to\mathcal H_{\mathcal K}.
\label{eq:lifting_covariance_operator}
\end{equation}
Fractional powers of \(T_{\mathcal K}\) are defined through this tensor
extension. Assume that a reference lifting \(\mathcal K_\star\) satisfies the source condition\cite{tuo2020improved}
\begin{equation}
\mathcal K_\star
=(T_{\mathcal K})^{\beta_{\mathcal K}}G_{\mathcal K},
\qquad
0<\beta_{\mathcal K}\le1,
\qquad
\|G_{\mathcal K}\|_{\mathcal H_{\mathcal K}}
\le B_{\mathcal K}.
\label{eq:lifting_source_condition}
\end{equation}

For analysis only, let \(\mathcal K_n^\circ\) denote the KRR estimator
obtained by replacing the empirical inputs \(\phi_i\) in
\eqref{eq:residual_lifting_krr} with the ideal inputs \(\phi_i^\star\),
while retaining the same residual targets:
\begin{equation}
\mathcal K_n^\circ
:=
\arg\min_{F\in\mathcal H_{\mathcal K}}
\left\{
\frac1{n_s}\sum_{i=1}^{n_s}
\|\mathbf r_i-F(\phi_i^\star)\|_{h,\tau}^2
+\lambda_{\mathcal K}\|F\|_{\mathcal H_{\mathcal K}}^2
\right\}.
\label{eq:ideal_coordinate_lifting_estimator}
\end{equation}
Its training set modeling mismatch is
\begin{equation}
E_{\rm l}
:=
\left[
\frac1{n_s}\sum_{i=1}^{n_s}
\|\mathbf r_i-\mathcal K_\star(\phi_i^\star)\|_{h,\tau}^2
\right]^{1/2}.
\label{eq:lifting_modeling_mismatch}
\end{equation}

\begin{assumption}[Fixed-design lifting estimate]
\label{ass:conditional_lifting_krr_oracle}
For every realized training set satisfying the preceding source and
regularity conditions,
\begin{equation}
\|\mathcal K_n^\circ-\mathcal K_\star\|_{\mathcal H_{\mathcal K}}
\le
C_{\mathcal K}
\left[
B_{\mathcal K}\lambda_{\mathcal K}^{\beta_{\mathcal K}}
+\frac{E_{\rm l}}{\sqrt{\lambda_{\mathcal K}}}
\right].
\label{eq:fixed_design_lifting_oracle}
\end{equation}
Any additional approximation term arising from the empirical covariance
operator or the use of a fixed training design in the selected KRR theorem is understood to be included in the right-hand side.
\end{assumption}

Define the aligned forms of the learned lifting map and the learned map from parameters to residual coordinates by
\begin{equation}
\widehat{\mathcal K}(y)
:=
\widehat{\mathcal K}_{\rm p}(J_\Phi y),
\qquad
\widehat\phi(\mu)
:=
J_\Phi\widehat\phi_{\rm p}(\mu).
\label{eq:aligned_lifting_and_parameter_map}
\end{equation}
Because \(J_\Phi^2=I_q\),
\begin{equation}
\widehat{\mathcal K}\bigl(\widehat\phi(\mu)\bigr)
=
\widehat{\mathcal K}_{\rm p}
\bigl(\widehat\phi_{\rm p}(\mu)\bigr).
\label{eq:alignment_preserves_reconstruction}
\end{equation}
Thus, alignment changes neither the algorithm nor its reconstruction.
Moreover, \eqref{eq:lifting_kernel_sign_invariance} implies that composition
with \(J_\Phi\) is an isometry of \(\mathcal H_{\mathcal K}\), so
\(\widehat{\mathcal K}\) is precisely the KRR estimator trained at the
aligned inputs \(\phi_i^{\rm al}\).

\begin{theorem}[Residual lifting error]
\label{thm:hilbert_lifting_error}
Under Proposition~\ref{prop:uniform_residual_coordinate_error} and
Assumption~\ref{ass:conditional_lifting_krr_oracle}, on
\(\mathcal A_{\rm op}\cap\mathcal A_{\rm sp}\),
\begin{equation}
\|\widehat{\mathcal K}-\mathcal K_\star\|_{\mathcal H_{\mathcal K}}
\le E_{\mathcal K},
\label{eq:residual_lifting_error}
\end{equation}
where
\begin{equation}
E_{\mathcal K}
:={}
C_{\mathcal K}
\left[
B_{\mathcal K}\lambda_{\mathcal K}^{\beta_{\mathcal K}}
+\frac{E_{\rm l}}{\sqrt{\lambda_{\mathcal K}}}
\right] +
R_{m}L_{\rm f}
\left[
\frac{E_{\rm b}}{\lambda_{\mathcal K}}
+\frac{2\kappa_{\mathcal K}^2E_{\rm b}}
{\lambda_{\mathcal K}^2}
\right], R_{m}:=\max_{1\le i\le n_s}\|\mathbf r_i\|_{h,\tau}.
\label{eq:explicit_lifting_error}
\end{equation}
\end{theorem}

\begin{proof}
Write \(k_\phi:=k_{\mathcal K}(\cdot,\phi)\). For an input set
\(X=\{\xi_i\}_{i=1}^{n_s}\), define
\[
T_X
:=
\left(
\frac1{n_s}\sum_{i=1}^{n_s}
k_{\xi_i}\otimes k_{\xi_i}
\right)
\otimes I_{\mathbb R^N_{h,\tau}},
\qquad
g_X
:=
\frac1{n_s}\sum_{i=1}^{n_s}
k_{\xi_i}\otimes\mathbf r_i.
\]
The associated KRR solution is
\((T_X+\lambda_{\mathcal K}I)^{-1}g_X\). Let
\(X_{\rm al}:=\{\phi_i^{\rm al}\}\) and
\(X_\star:=\{\phi_i^\star\}\).  By
\eqref{eq:lifting_kernel_bound},
\eqref{eq:lifting_feature_lipschitz}, and
\eqref{eq:embedding_error},
\[
\|g_{X_{\rm al}}-g_{X_\star}\|
\le L_{\rm f}E_{\rm b}R_{m}, \quad \|T_{X_{\rm al}}-T_{X_\star}\|_{\rm op}
\le2\kappa_{\mathcal K}L_{\rm f}E_{\rm b}, \quad \|g_{X_\star}\|\le\kappa_{\mathcal K}R_{m}.
\]
Here, the second estimate follows from
\(k_a\otimes k_a-k_b\otimes k_b
=(k_a-k_b)\otimes k_a+k_b\otimes(k_a-k_b)\).

Set \(A:=T_{X_{\rm al}}+\lambda_{\mathcal K}I\) and
\(B:=T_{X_\star}+\lambda_{\mathcal K}I\). Using
\(A^{-1}-B^{-1}=A^{-1}(B-A)B^{-1}\) and
\(\|A^{-1}\|_{\rm op},\|B^{-1}\|_{\rm op}
\le\lambda_{\mathcal K}^{-1}\), we obtain
\begin{align*}
\|\widehat{\mathcal K}-\mathcal K_n^\circ\|
&\le
\|A^{-1}(g_{X_{\rm al}}-g_{X_\star})\|
+\|A^{-1}(T_{X_\star}-T_{X_{\rm al}})
B^{-1}g_{X_\star}\|\\
&\le
R_{m}L_{\rm f}
\left[
\frac{E_{\rm b}}{\lambda_{\mathcal K}}
+\frac{2\kappa_{\mathcal K}^2E_{\rm b}}
{\lambda_{\mathcal K}^2}
\right].
\end{align*}
Finally, the triangle inequality and
\eqref{eq:fixed_design_lifting_oracle} yield
\(\|\widehat{\mathcal K}-\mathcal K_\star\|_{\mathcal H_{\mathcal K}}
\le E_{\mathcal K}\), proving
\eqref{eq:residual_lifting_error}.
\end{proof}

\subsection{Ideal parametric coordinates and direct-KRR error}
\label{subsec:parameter_latent_error}

Assume that \(\mathbf r_\star(\mu)\in\mathcal U_{\mathcal M}\) for every
\(\mu\in\mathcal D\), and extend the ideal residual coordinate to the
parameter domain by
\begin{equation}
\phi_\star^\mu(\mu)
:=
\left(
(\vartheta_1^\star)^{t_{\rm res}}
\psi_1^\star\!\left(\Pi_{\mathcal M}\mathbf r_\star(\mu)\right),
\ldots,
(\vartheta_q^\star)^{t_{\rm res}}
\psi_q^\star\!\left(\Pi_{\mathcal M}\mathbf r_\star(\mu)\right)
\right)^\top.
\label{eq:ideal_parameter_residual_coordinate}
\end{equation}
In particular, \(\phi_\star^\mu(\mu_i)=\phi_i^\star\).

Let \(\rho_{\mathcal D}\) be the parameter sampling probability measure. For \(\ell\in\{z,\phi\}\), let \(k_{\mu,\ell}\) be the scalar kernel underlying the separable vector-valued RKHS used in
Section~\ref{subsubsec:direct_parameter_regression}, and define
\begin{align}
(T_{\mu,\ell}f)(\cdot)
&:=
\int_{\mathcal D}
k_{\mu,\ell}(\cdot,\nu)f(\nu)\,d\rho_{\mathcal D}(\nu),
\label{eq:parameter_covariance_operator}\\
\mathcal N_{\mu,\ell}(\lambda)
&:=
\operatorname{tr}
\left[
T_{\mu,\ell}(T_{\mu,\ell}+\lambda I)^{-1}
\right].
\label{eq:parameter_effective_dimension}
\end{align}
The operator \(T_{\mu,\ell}\) and its fractional powers act componentwise on the corresponding vector-valued RKHS. Set
\[
\ell_\star
:=
\begin{cases}
z_\star, & \ell=z,\\
\phi_\star^\mu, & \ell=\phi,
\end{cases}
\qquad
\widehat\ell
:=
\begin{cases}
\widehat z, & \ell=z,\\
J_\Phi\widehat\phi_{\rm p}, & \ell=\phi.
\end{cases}
\]
The corresponding aligned training labels are \(z_i\) and
\(J_\Phi\phi_i\). Their root mean square errors relative to the ideal labels are
\begin{equation}
E_{{\rm obs},z}=0,
\qquad
E_{{\rm obs},\phi}
:=
\left[
\frac1{n_s}\sum_{i=1}^{n_s}
\|\phi_i^{\rm al}-\phi_i^\star\|_2^2
\right]^{1/2}
\le E_{\rm b}.
\label{eq:parameter_label_errors}
\end{equation}

Assume the source conditions\cite{caponnetto2007optimal}
\begin{equation}
\ell_\star
=T_{\mu,\ell}^{\beta_\ell}g_\ell,
\qquad
0<\beta_\ell\le1,
\qquad
\|g_\ell\|_{\mathcal H_{\mu,\ell}}
\le B_\ell,
\qquad
\ell\in\{z,\phi\}.
\label{eq:parameter_source_condition}
\end{equation}
Because these targets depend on the representation learned from the same
training sample. We therefore state explicitly the conditional KRR estimate required for the parameter maps.

\begin{assumption}[Conditional KRR estimate for the parameter maps]
\label{ass:parameter_latent_krr_oracle}
For every \(\delta_\ell\in(0,1)\), there is an event
\(\mathcal A_{\mu,\ell}\), with
\(\mathbb P(\mathcal A_{\mu,\ell}^c)\le\delta_\ell\), on which
\begin{equation}
\|\widehat\ell-\ell_\star\|_{L^2_{\rho_{\mathcal D}}}
\le
E_{\mu,\ell}(\delta_\ell)
:=
C_\ell
\left[
B_\ell\lambda_{\mu,\ell}^{\beta_\ell}
+\sigma_\ell
\sqrt{
\frac{
\mathcal N_{\mu,\ell}(\lambda_{\mu,\ell})
+\log(2/\delta_\ell)}
{n_s\lambda_{\mu,\ell}}
}
+\frac{E_{{\rm obs},\ell}}
{\sqrt{\lambda_{\mu,\ell}}}
\right].
\label{eq:parameter_latent_oracle}
\end{equation}
Here \(\sigma_\ell\) denotes the stochastic scale entering the corresponding
KRR estimate, and \(\sigma_\ell\), \(\ell\in\{z,\phi\}\), are assumed to be
uniformly bounded. If a fixed-design result is used, the population effective
dimension in \eqref{eq:parameter_latent_oracle} may be replaced by its
empirical counterpart, together with the corresponding sampling term.
\end{assumption}

\begin{theorem}[Direct-KRR latent error]
\label{thm:parameter_latent_krr_error}
Under Assumption~\ref{ass:parameter_latent_krr_oracle},
\begin{equation}
\|\widehat z-z_\star\|_{L^2_{\rho_{\mathcal D}}}
\le E_{\mu,z}(\delta_z),
\qquad
\|\widehat\phi-\phi_\star^\mu\|_{L^2_{\rho_{\mathcal D}}}
\le E_{\mu,\phi}(\delta_\phi)
\label{eq:joint_direct_krr_latent_error}
\end{equation}
on an event with a failure probability of at most
\(\delta_z+\delta_\phi\).
\end{theorem}

\begin{proof}
Let
\(\mathcal A_\mu:=
\mathcal A_{\mu,z}\cap\mathcal A_{\mu,\phi}\).
By Assumption~\ref{ass:parameter_latent_krr_oracle}, both estimates in
\eqref{eq:joint_direct_krr_latent_error} hold simultaneously on
\(\mathcal A_\mu\). Moreover, without requiring independence,
\[
\mathbb P(\mathcal A_\mu^c)
=
\mathbb P\!\left(
\mathcal A_{\mu,z}^c\cup\mathcal A_{\mu,\phi}^c
\right)
\le
\mathbb P(\mathcal A_{\mu,z}^c)
+\mathbb P(\mathcal A_{\mu,\phi}^c)
\le\delta_z+\delta_\phi.
\]
Thus \(\mathbb P(\mathcal A_\mu)\ge1-\delta_z-\delta_\phi\), which proves
the claim.
\end{proof}

\subsection{Latent error for continuation with smoothing splines}
\label{subsec:continuation_spline_error}

Suppose that \(\mathcal D=[a,b]\) and that the latent coordinate maps are constructed by continuation followed by smoothing spline fitting, as
described in Section~\ref{subsubsec:one_dimensional_continuation}. Define
\[
Q_z:=I_r,
\qquad
Q_\phi:=J_\Phi,
\qquad
\widetilde y_i^{\,\ell,{\rm al}}
:=Q_\ell\widetilde y_i^\ell.
\]
Because the Euclidean data norm and the spline roughness penalty are
orthogonally invariant, \(Q_\ell\widehat y_{\rm s}^{\,\ell}\) is exactly
the spline fitted to the aligned labels
\(\widetilde y_i^{\,\ell,{\rm al}}\). Thus, this change of coordinates also leaves the PPMD reconstruction unchanged.

For \(i=1,\ldots,N_p\), set
\begin{equation}
e_i^\ell
:=
\widetilde y_i^{\,\ell,{\rm al}}-\ell_\star(\mu_i).
\label{eq:spline_label_error}
\end{equation}
Using the weights and \(N_{\rm e}^\ell\) already defined in the equation \eqref{eq:weight}, introduce
\begin{equation}
E_{{\rm h},\ell} :=
\left[
\frac1{N_{\rm e}^\ell}
\sum_{i=1}^{n_s}\|e_i^\ell\|_2^2
\right]^{1/2},\qquad E_{{\rm c},\ell}:=
\left[
\frac{\gamma_\ell}{N_{\rm e}^\ell}
\sum_{i=n_s+1}^{N_p}\|e_i^\ell\|_2^2
\right]^{1/2}.
\label{eq:continuation_label_error}
\end{equation}
The two index sets are disjoint; hence,
\begin{equation}
\left(E_{{\rm d},\ell}^{\rm s}\right)^2
:=
\frac1{N_{\rm e}^\ell}
\sum_{i=1}^{N_p}
\omega_i^\ell\|e_i^\ell\|_2^2
=
\left(E_{{\rm h},\ell}\right)^2
+\left(E_{{\rm c},\ell}\right)^2.
\label{eq:spline_data_error_decomposition}
\end{equation}
For the original high-fidelity labels,
\begin{equation}
E_{{\rm h},z}=0,
\qquad
E_{{\rm h},\phi}
\le
\left(\frac{n_s}{N_{\rm e}^{\phi}}\right)^{1/2}
E_{\rm b}.
\label{eq:high_fidelity_embedding_contribution}
\end{equation}

To make the accumulation of continuation error explicit, let
\(C_\ell^\Delta\) be the exact shift map in the aligned latent coordinates,
and let \(\widehat C_\ell^\Delta\) denote the aligned learned continuation
map. The exact shift map and the aligned learned continuation map satisfy
\[
C_\ell^\Delta\bigl(\ell_\star(\mu_i)\bigr)
=
\ell_\star(\mu_{i+1}),\qquad \widehat C_\ell^\Delta
:=
Q_\ell\widehat T_\ell Q_\ell^\top,
\qquad i=n_s,\ldots,N_p-1.
\]
If
\[
\operatorname{Lip}(C_\ell^\Delta)\le L_\ell,
\qquad
\sup_x
\|\widehat C_\ell^\Delta(x)-C_\ell^\Delta(x)\|_2
\le\eta_\ell,
\]
then successive continuation errors satisfy
\begin{equation}
\|e_{n_s+j}^\ell\|_2
\le
L_\ell^j\|e_{n_s}^\ell\|_2
+\eta_\ell\sum_{k=0}^{j-1}L_\ell^k,
\qquad
1\le j\le N_p-n_s.
\label{eq:continuation_recursion}
\end{equation}
For a uniformly bounded number of continuation steps, it is sufficient that \(\|e_{n_s}^\ell\|_2\to0\), \(\eta_\ell\to0\), and
\(\sup L_\ell<\infty\). If the number of continuation steps increases, it is sufficient that
\[
\|e_{n_s}^\ell\|_2\to0,\qquad
\eta_\ell\to0,\qquad
L_\ell\le L_\ast<1.
\]
In the neutral case \(L_\ell=1\), one requires
\(\|e_{n_s}^\ell\|_2\to0\) and
\((N_p-n_s)\eta_\ell\to0\). 

Let
\[
\widehat\ell_{\rm s}
:=
Q_\ell\widehat y_{\rm s}^{\,\ell},
\]
and let \(s_\ell^\circ\) be the spline produced by the same knots, weights,
and regularization parameter from exact labels
\(\ell_\star(\mu_i)\).  Assume the spline solution operator is stable:
\begin{equation}
\|\widehat\ell_{\rm s}-s_\ell^\circ\|_{L^2(a,b)}
\le
C_{{\rm t},\ell}E_{{\rm d},\ell}^{\rm s}.
\label{eq:spline_stability}
\end{equation}
Let
\[
h_\mu:=\max_{1\le i<N_p}(\mu_{i+1}-\mu_i).
\]

\begin{lemma}[Spline latent error]
\label{lem:inline_spline_latent_error}
If \(\ell_\star\in H^{m_\ell}(a,b;\mathbb R^{d_\ell})\) and the exact-data spline satisfies the standard approximation estimate of order \(m_\ell\)\cite{wahba1990spline,de2001calculation,cox1984multivariate}, then
\begin{equation}
\|\widehat\ell_{\rm s}-\ell_\star\|_{L^2(a,b)}
\le{}
C_{{\rm a},\ell}
\left[
h_\mu^{m_\ell}
\|\ell_\star\|_{H^{m_\ell}(a,b)}
+\alpha_\ell^{1/2}
\|\ell_\star\|_{H^2(a,b)}
\right]+
C_{{\rm t},\ell}E_{{\rm d},\ell}^{\rm s}.
\label{eq:spline_latent_error}
\end{equation}
\end{lemma}

\begin{proof}
Let \(s_\ell^\circ\) denote the spline constructed with the same knots,
weights, and regularization parameter as \(\widehat\ell_{\rm s}\), but using
the exact values \(\ell_\star(\mu_i)\). By the triangle inequality,
the spline-stability estimate \eqref{eq:spline_stability}, and the assumed exact data approximation estimate,
\begin{align*}
\|\widehat\ell_{\rm s}-\ell_\star\|_{L^2(a,b)}
&\le
\|\widehat\ell_{\rm s}-s_\ell^\circ\|_{L^2(a,b)}
+\|s_\ell^\circ-\ell_\star\|_{L^2(a,b)}\\
&\le
C_{{\rm t},\ell}E_{{\rm d},\ell}^{\rm s}
+C_{{\rm a},\ell}
\left[
h_\mu^{m_\ell}\|\ell_\star\|_{H^{m_\ell}(a,b)}
+\alpha_\ell^{1/2}\|\ell_\star\|_{H^2(a,b)}
\right].
\end{align*}
The vector-valued estimate follows componentwise and therefore proves
\eqref{eq:spline_latent_error}.
\end{proof}

If
\(d\rho_{\mathcal D}(\mu)=w_{\mathcal D}(\mu)\,d\mu\) with
\(0\le w_{\mathcal D}\le c_\rho\) almost everywhere, define
\begin{equation}
E_{\ell,\rho}^{\rm s}
:=
c_\rho^{1/2}
\left\{
C_{{\rm a},\ell}
\left[
h_\mu^{m_\ell}\|\ell_\star\|_{H^{m_\ell}}
+\alpha_\ell^{1/2}\|\ell_\star\|_{H^2}
\right]
+C_{{\rm t},\ell}E_{{\rm d},\ell}^{\rm s}
\right\}.
\label{eq:spline_parameter_measure_error}
\end{equation}
Then
\begin{equation}
\|\widehat\ell_{\rm s}-\ell_\star\|
_{L^2_{\rho_{\mathcal D}}(\mathcal D)}
\le E_{\ell,\rho}^{\rm s}.
\label{eq:spline_latent_rho_error}
\end{equation}

\subsection{Deterministic and high probability PPMD bounds}
\label{subsec:ppmd_main_error}

The out of sample residual modeling defect is
\begin{equation}
E_{\rm m}(\mu)
:=
\left\|
\mathbf r_\star(\mu)
-\mathcal K_\star\bigl(\phi_\star^\mu(\mu)\bigr)
\right\|_{h,\tau},
\qquad
E_{{\rm m},\rho}
:=
\|E_{\rm m}\|_{L^2_{\rho_{\mathcal D}}(\mathcal D)}.
\label{eq:out_of_sample_modeling_defect}
\end{equation}
It includes residual coordinate truncation, possible noninjectivity of the spectral coordinates, and any off manifold modeling error. Assume that
\(\mathcal K_\star\) is uniformly Lipschitz on \(\mathcal Z\):
\begin{equation}
\|\mathcal K_\star(\phi)-\mathcal K_\star(\eta)\|_{h,\tau}
\le L_{\mathcal K}\|\phi-\eta\|_2.
\label{eq:reference_lifting_lipschitz}
\end{equation}

\begin{theorem}[Deterministic PPMD error]
\label{thm:deterministic_ppmd_error}
On any realization for which
\eqref{eq:residual_lifting_error} holds, the pointwise reconstruction error
satisfies
\begin{equation}
\begin{split}
\left\|
\mathcal S(\mu)-\widehat{\mathcal S}_{r,h,\tau}(\mu)
\right\|&_{\mathcal X_T}
\le{}
E_{\rm d}(h,\tau)\\
+
L_{\mathcal N^{-1}}&
\left[E_{\rm m}(\mu)
+\kappa_{\mathcal K}E_{\mathcal K}
+\|\widehat z(\mu)-z_\star(\mu)\|_2
+L_{\mathcal K}
\|\widehat\phi(\mu)-\phi_\star^\mu(\mu)\|_2\right].
\end{split}
\label{eq:deterministic_ppmd_error}
\end{equation}
Consequently,
\begin{equation}
\begin{split}
\left\|
\mathcal S-\widehat{\mathcal S}_{r,h,\tau}
\right\|&_{L^2_{\rho_{\mathcal D}}(\mathcal D;\mathcal X_T)}
\le{}
E_{\rm d}(h,\tau)\\
+
L_{\mathcal N^{-1}}&
\left[
E_{{\rm m},\rho}
+\kappa_{\mathcal K}E_{\mathcal K}
+\|\widehat z-z_\star\|_{L^2_{\rho_{\mathcal D}}}
+L_{\mathcal K}
\|\widehat\phi-\phi_\star^\mu\|_{L^2_{\rho_{\mathcal D}}}
\right].
\end{split}
\label{eq:integrated_deterministic_ppmd_error}
\end{equation}
\end{theorem}

\begin{proof}
Fix \(\mu\in\mathcal D\). Adding and subtracting the reconstructed
fully discrete solution and using \eqref{eq:uniform_discretization_error}
gives
\begin{align*}
\|\mathcal S(\mu)-\widehat{\mathcal S}_{r,h,\tau}(\mu)\|_{\mathcal X_T}
&\le E_{\rm d}(h,\tau)
+\|\mathcal E_{h,\tau}
(\mathbf u_{h,\tau}-\widehat{\mathbf u}_{r,h,\tau})\|_{\mathcal X_T}\\
&\le E_{\rm d}(h,\tau)
+L_{\mathcal N^{-1}}
\|\bar{\mathbf u}_{h,\tau}
-\widehat{\bar{\mathbf u}}_{r,h,\tau}\|_{h,\tau}.
\end{align*}
Here, the second inequality follows from the isometry of
\(\mathcal E_{h,\tau}\) and
\[
\mathbf u_{h,\tau}-\widehat{\mathbf u}_{r,h,\tau}
=
B^{-1}\bigl(
\bar{\mathbf u}_{h,\tau}
-\widehat{\bar{\mathbf u}}_{r,h,\tau}
\bigr),
\]
where the affine shifts cancel.

Using the exact decomposition
\(\bar{\mathbf u}_{h,\tau}=\Psi_rz_\star+\mathbf r_\star\),
the aligned reconstruction
\(\widehat{\bar{\mathbf u}}_{r,h,\tau}
=\Psi_r\widehat z+\widehat{\mathcal K}(\widehat\phi)\), and inserting
\(\mathcal K_\star(\phi_\star^\mu)\) and
\(\mathcal K_\star(\widehat\phi)\), we obtain
\begin{align*}
\|\bar{\mathbf u}_{h,\tau}
-\widehat{\bar{\mathbf u}}_{r,h,\tau}\|_{h,\tau}
\le{}&
\|\Psi_r(z_\star-\widehat z)\|_{h,\tau}
+\|\mathbf r_\star-\mathcal K_\star(\phi_\star^\mu)\|_{h,\tau}\\
&+\|\mathcal K_\star(\phi_\star^\mu)
-\mathcal K_\star(\widehat\phi)\|_{h,\tau}
+\|(\mathcal K_\star-\widehat{\mathcal K})(\widehat\phi)\|_{h,\tau}\\
\le{}&
\|z_\star-\widehat z\|_2+E_{\rm m}(\mu)
+L_{\mathcal K}\|\phi_\star^\mu-\widehat\phi\|_2
+\kappa_{\mathcal K}E_{\mathcal K}.
\end{align*}
Indeed, the last line uses the weighted orthonormality of \(\Psi_r\), the
definition of \(E_{\rm m}\), the Lipschitz continuity of
\(\mathcal K_\star\), and the RKHS evaluation bound
\[
\|(\mathcal K_\star-\widehat{\mathcal K})(\widehat\phi)\|_{h,\tau}
\le
\sqrt{k_{\mathcal K}(\widehat\phi,\widehat\phi)}
\|\mathcal K_\star-\widehat{\mathcal K}\|_{\mathcal H_{\mathcal K}}
\le\kappa_{\mathcal K}E_{\mathcal K}.
\]
Substitution proves \eqref{eq:deterministic_ppmd_error}. Finally, taking the
\(L^2_{\rho_{\mathcal D}}(\mathcal D)\) norm and applying Minkowski's
inequality, together with
\(\|E_{\rm m}\|_{L^2_{\rho_{\mathcal D}}}=E_{{\rm m},\rho}\), yields
\eqref{eq:integrated_deterministic_ppmd_error}.
\end{proof}

\begin{theorem}[High-probability PPMD error]
\label{thm:main_ppmd_probabilistic_error}
For the direct-KRR variant, with probability at least $1-\delta$, $\delta
:=
\delta_{\rm op}+\delta_{\rm sp}+\delta_z+\delta_\phi$,
\begin{equation}
\begin{split}
\left\|
\mathcal S-\widehat{\mathcal S}_{r,h,\tau}
\right\|_{L^2_{\rho_{\mathcal D}}(\mathcal D;\mathcal X_T)}
\le{}&
E_{\rm d}(h,\tau)\\
+
L_{\mathcal N^{-1}}&
\left[
E_{{\rm m},\rho}
+\kappa_{\mathcal K}E_{\mathcal K}
+E_{\mu,z}(\delta_z)
+L_{\mathcal K}E_{\mu,\phi}(\delta_\phi)
\right].
\end{split}
\label{eq:direct_krr_high_probability_ppmd_error}
\end{equation}
For the continuation and smoothing spline construction, suppose that the
continuation approximation and spline stability assumptions hold for every admissible training realization. Then, with probability at least
\(1-\delta_{\rm op}-\delta_{\rm sp}\),
\begin{equation}
\begin{split}
\left\|
\mathcal S-\widehat{\mathcal S}_{r,h,\tau}
\right\|_{L^2_{\rho_{\mathcal D}}(\mathcal D;\mathcal X_T)}
\le{}&
E_{\rm d}(h,\tau)\\
&+
L_{\mathcal N^{-1}}
\left[
E_{{\rm m},\rho}
+\kappa_{\mathcal K}E_{\mathcal K}
+E_{z,\rho}^{\rm s}
+L_{\mathcal K}E_{\phi,\rho}^{\rm s}
\right].
\end{split}
\label{eq:spline_high_probability_ppmd_error}
\end{equation}
\end{theorem}

\begin{proof}
For the direct-KRR branch, define
\[
\mathcal A_{\rm dir}
:=
\mathcal A_{\rm op}\cap\mathcal A_{\rm sp}
\cap\mathcal A_{\mu,z}\cap\mathcal A_{\mu,\phi}.
\]
On \(\mathcal A_{\rm op}\cap\mathcal A_{\rm sp}\),
Proposition~\ref{prop:uniform_residual_coordinate_error} controls
\(E_{\rm b}\), and Theorem~\ref{thm:hilbert_lifting_error} gives
\(\|\widehat{\mathcal K}-\mathcal K_\star\|_{\mathcal H_{\mathcal K}}
\le E_{\mathcal K}\). On
\(\mathcal A_{\mu,z}\cap\mathcal A_{\mu,\phi}\),
Theorem~\ref{thm:parameter_latent_krr_error} gives
\[
\|\widehat z-z_\star\|_{L^2_{\rho_{\mathcal D}}}
\le E_{\mu,z}(\delta_z),
\qquad
\|\widehat\phi-\phi_\star^\mu\|_{L^2_{\rho_{\mathcal D}}}
\le E_{\mu,\phi}(\delta_\phi).
\]
Substitution into \eqref{eq:integrated_deterministic_ppmd_error} proves
\eqref{eq:direct_krr_high_probability_ppmd_error}. Moreover, no independence is required, since the union bound gives
\[
\mathbb P(\mathcal A_{\rm dir}^c)
\le
\delta_{\rm op}+\delta_{\rm sp}+\delta_z+\delta_\phi
=\delta.
\]

For the construction based on continuation followed by smoothing spline
fitting, let \(\mathcal A_{\rm s}:=\mathcal A_{\rm op}\cap\mathcal A_{\rm sp}\). On this event, the same spectral and lifting bounds hold, while \eqref{eq:spline_latent_rho_error} gives
\[
\|\widehat z-z_\star\|_{L^2_{\rho_{\mathcal D}}}
\le E_{z,\rho}^{\rm s},
\qquad
\|\widehat\phi-\phi_\star^\mu\|_{L^2_{\rho_{\mathcal D}}}
\le E_{\phi,\rho}^{\rm s}.
\]
Inserting these estimates into
\eqref{eq:integrated_deterministic_ppmd_error} proves
\eqref{eq:spline_high_probability_ppmd_error} and
\(\mathbb P(\mathcal A_{\rm s}^c)
\le\delta_{\rm op}+\delta_{\rm sp}\).
\end{proof}

If the continuation or spline estimates hold only on events of high
probability, their failure probabilities must also be included in the
corresponding union bound.

\subsection{Consistency}
\label{subsec:ppmd_consistency}

\begin{theorem}[Consistency]
\label{thm:ppmd_consistency}
Consider a sequence indexed by \(n\), with
\(h_n,\tau_n\to0\) and \(n_s\to\infty\).  Assume that the geometric,
spectral, source, discretization, and stability hypotheses above hold along the sequence, and that, for every fixed choice of confidence levels,
\begin{equation}
E_{{\rm d},n}\to0,
\quad
\overline E_{{\rm m},n}\to0,
\quad
B_{\mathcal K,n}\lambda_{\mathcal K,n}^{\beta_{\mathcal K,n}}\to0,
\quad
\frac{\overline E_{{\rm l},n}}
{\sqrt{\lambda_{\mathcal K,n}}}\to0,
\quad
\frac{E_{{\rm b},n}}{\lambda_{\mathcal K,n}^2}\to0,
\label{eq:common_consistency_conditions}
\end{equation}
where, almost surely,
\[
E_{{\rm m},\rho,n}\le\overline E_{{\rm m},n},
\qquad
E_{{\rm l},n}\le\overline E_{{\rm l},n}.
\]
Assume also that \(0<\lambda_{\mathcal K,n}\le1\) eventually holds, and that all multiplicative quantities appearing in the preceding error bounds, including \(L_{\mathcal N^{-1}}\), \(\kappa_{\mathcal K}\), \(L_{\mathcal K}\), \(L_{\rm f}\), \(R_{m}\), and \(\sigma_\ell\), remain uniformly bounded along the approximation sequence.

For the direct-KRR variant, assume in addition, for
\(\ell\in\{z,\phi\}\),
\begin{equation}
B_{\ell,n}\lambda_{\mu,\ell,n}^{\beta_{\ell,n}}\to0,
\qquad
\frac{
\mathcal N_{\mu,\ell,n}(\lambda_{\mu,\ell,n})
+\log(2/\delta_\ell)}
{n_s\lambda_{\mu,\ell,n}}
\to0,
\qquad
\frac{E_{{\rm obs},\ell,n}}
{\sqrt{\lambda_{\mu,\ell,n}}}
\to0.
\label{eq:direct_krr_consistency_conditions}
\end{equation}
For the continuation-spline variant, replace
\eqref{eq:direct_krr_consistency_conditions} by
\begin{equation}
h_{\mu,n}^{m_\ell}
\|\ell_{\star,n}\|_{H^{m_\ell}}\to0,
\quad
\alpha_{\ell,n}^{1/2}
\|\ell_{\star,n}\|_{H^2}\to0,
\quad
E_{{\rm d},\ell,n}^{\rm s}\xrightarrow{a.s.}0,
\quad
\ell\in\{z,\phi\},
\label{eq:spline_consistency_conditions}
\end{equation}
with the continuation errors controlled, for example, by one of the
sufficient conditions following \eqref{eq:continuation_recursion}.  Then,
for either branch,
\begin{equation}
\left\|
\mathcal S-\widehat{\mathcal S}_{r,h_n,\tau_n}
\right\|_{L^2_{\rho_{\mathcal D}}(\mathcal D;\mathcal X_T)}
\xrightarrow{\mathbb P}0.
\label{eq:ppmd_consistency_conclusion}
\end{equation}
\end{theorem}

\begin{proof}
Fix a tolerance \(\epsilon>0\) and an arbitrary probability level
\(\eta>0\). Choose fixed confidence parameters whose sum is smaller than
\(\eta\). Under \eqref{eq:common_consistency_conditions}, every term in
\eqref{eq:explicit_lifting_error} tends to zero; the condition
\(E_{{\rm b},n}/\lambda_{\mathcal K,n}^2\to0\) also controls
\(E_{{\rm b},n}/\lambda_{\mathcal K,n}\) because
\(\lambda_{\mathcal K,n}\le1\) eventually. The assumptions associated with each construction imply that the corresponding errors in the latent coordinates converge to zero. Consequently, for all sufficiently large \(n\), the right side of either \eqref{eq:direct_krr_high_probability_ppmd_error} or
\eqref{eq:spline_high_probability_ppmd_error} is smaller than
\(\epsilon\) outside of an event whose probability is at most \(\eta\).
Since \(\eta\) is arbitrary, \eqref{eq:ppmd_consistency_conclusion} follows.
\end{proof}

\section{Numerical results for flow past a cylinder}
\label{sec:num_cylinder}

We consider incompressible flow past a cylinder. 
\[
\Omega =
(0,1.8)\times(0,0.41)
\setminus
\overline{B_{0.05}(0.2,0.2)}. 
\]
The density, reference inflow velocity, and cylinder diameter are
\(\rho=1\), \(U=1\), and \(D=0.1\), respectively. The dynamic viscosity
\(\mu\) is the model parameter, with $\operatorname{Re}(\mu)
=
\frac{\rho U D}{\mu}
=
\frac{0.1}{\mu}$. We use $\mu_i
=
2.0\times10^{-5}
+
(i-1)\times10^{-5}, i=1,\ldots,100$, corresponding approximately to \(99\le\operatorname{Re}\le5000\).

The inlet velocity is \((1,0)\), homogeneous no-slip conditions are imposed on the walls and cylinder, and the initial velocity is \((0.1,0.1)\). The full-order problem is discretized using a mesh with 3802 nodes. The solution is computed on \([0,15]\) with \(\Delta t=0.1\). Each trajectory contains 150 states at \(t_n=n\Delta t\) and \(n=1,\ldots,150\); the common initial state is omitted.

The first 90 parameters, \(2.0\times10^{-5}\le\mu\le9.1\times10^{-4}\), are used for training, and the remaining ten, \(9.2\times10^{-4}\le\mu\le1.01\times10^{-3}\), form an extrapolation set.
No extrapolation trajectory is used in normalization, dimension reduction, graph construction, regression, lifting, or hyperparameter selection.

PPMD is compared with two nonintrusive baselines, Proper Orthogonal Decomposition (POD) + KRR and POD + Autoencoder (AE) + KRR. For POD+KRR, the normalized trajectories are projected onto a
12-dimensional weighted POD space, and the corresponding POD coefficients
are predicted directly from \(\mu\) by KRR. The reconstructed trajectory is obtained by combining the predicted coefficients with the fixed POD basis.

POD+AE+KRR uses an eight-dimensional weighted POD representation for the
dominant linear component. An autoencoder is trained on the corresponding
POD residuals to obtain four nonlinear latent coordinates, and its decoder maps these coordinates back to the residual trajectory space. KRR is used to predict both the eight POD coefficients and the four autoencoder coordinates from \(\mu\). The decoded residual is then added to the POD reconstruction.

All three methods, therefore, use a total latent dimension of 12. PPMD and
POD+AE+KRR each use eight linear and four nonlinear coordinates, whereas
POD + KRR uses 12 linear coordinates. For PPMD, both the linear coordinates
and the residual spectral coordinates are predicted directly from \(\mu\)
by KRR, rather than by recursive continuation. All methods use the same training and extrapolation sets, normalization procedure, weighted trajectory norm, and procedure for selecting hyperparameters without using extrapolation data.

For any parameter \(\mu\), the relative trajectory error is
\begin{equation}
e_{h,\tau}(\mu)
:=
\frac{
\|
\mathbf u_{h,\tau}(\mu)
-
\widehat{\mathbf u}_{h,\tau}(\mu)
\|_{h,\tau}
}{
\|
\mathbf u_{h,\tau}(\mu)
\|_{h,\tau}
}.
\label{eq:cylinder_relative_trajectory_error}
\end{equation}
This norm incorporates the spatial discretization matrix and time quadrature
weights and is therefore distinct from the unweighted Euclidean norm of the
stacked coefficient vector. All results concern complete trajectories on the
fixed interval \([0,15]\), rather than autonomous time advancement beyond
that interval.

\subsection{Reconstruction and extrapolation results}
\label{subsec:cylinder_prediction_results}

Figure~\ref{fig:cylinder_training_extrapolation} separates reconstruction at
the training parameters from prediction outside the training interval.

\begin{figure}[t]
\centering
\begin{minipage}[t]{0.49\textwidth}
\centering
\includegraphics[width=\linewidth]
{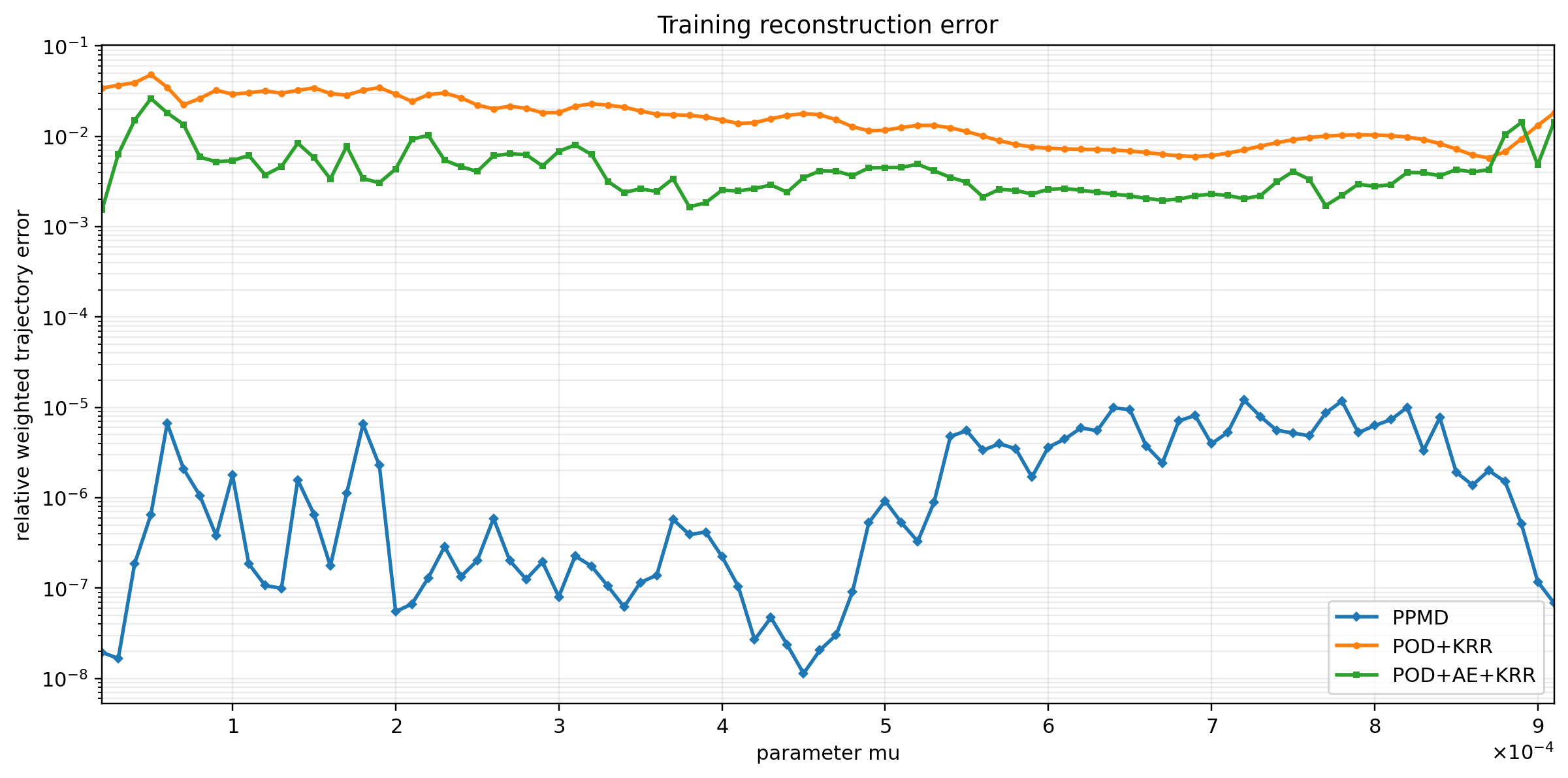}

{\small (a) Training reconstruction}
\end{minipage}
\hfill
\begin{minipage}[t]{0.49\textwidth}
\centering
\includegraphics[width=\linewidth]
{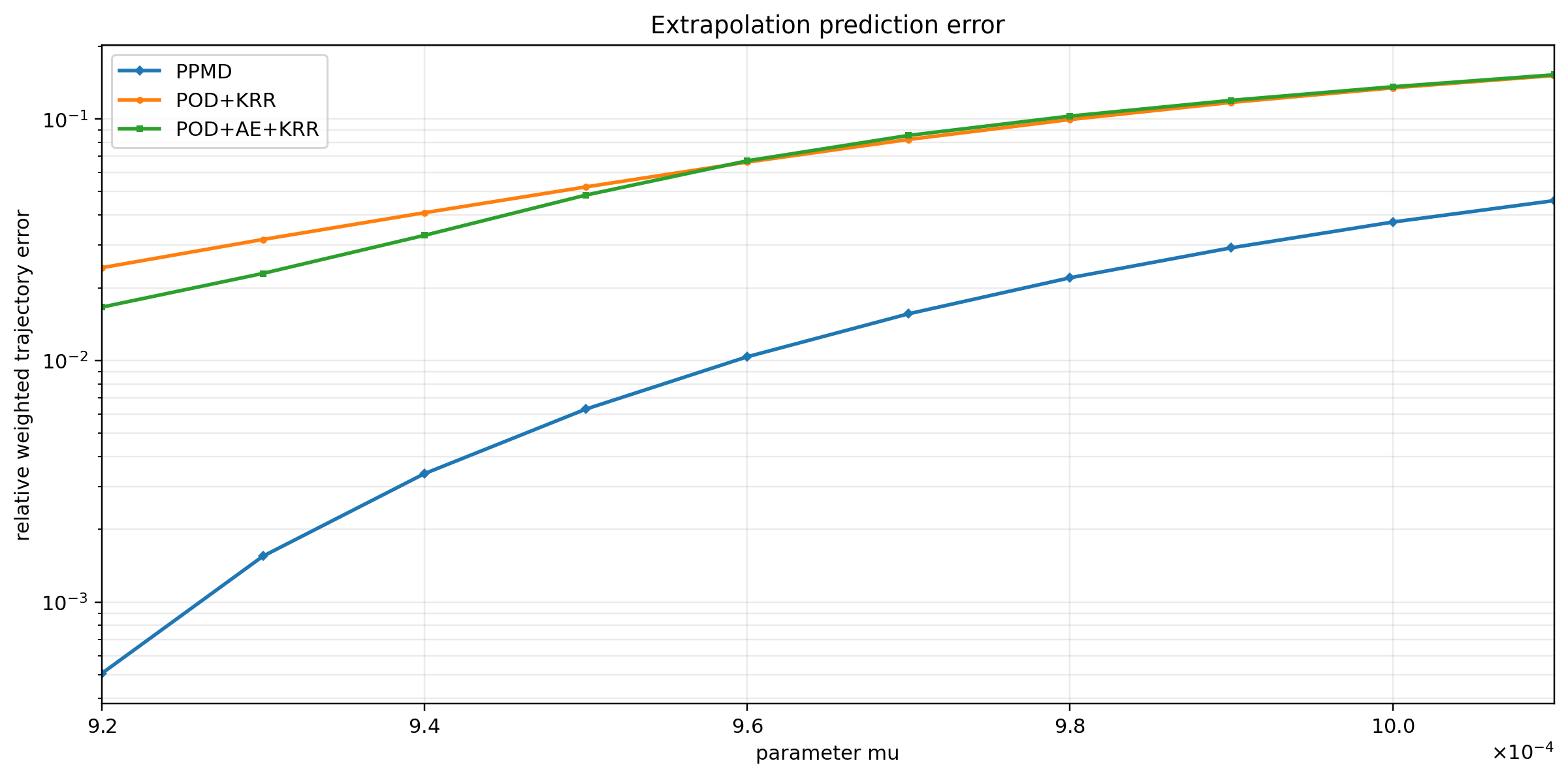}

{\small (b) Extrapolation prediction}
\end{minipage}
\caption{
Relative weighted trajectory errors for PPMD, POD+KRR, and POD+AE+KRR.
Panel (a) reports reconstruction at the 90 training parameters, and panel
(b) reports prediction at the ten extrapolation parameters. All methods use
12 latent coordinates; PPMD and POD+AE+KRR use eight linear and four
nonlinear coordinates.
}
\label{fig:cylinder_training_extrapolation}
\end{figure}

PPMD reconstructs the training trajectories with errors between
approximately \(10^{-8}\) and \(10^{-5}\), whereas the two comparison methods
have substantially larger training errors. This demonstrates the
representation capacity of PPMD but is not, by itself, evidence of
generalization.

On the extrapolation set, PPMD gives the smallest error at all ten parameters.
Its error increases from approximately \(5\times10^{-4}\) near the training
boundary to \(4.6\times10^{-2}\) at the most distant parameter. In
comparison, POD+KRR and POD+AE+KRR have errors of order \(10^{-2}\) near the
boundary and approximately \(1.5\times10^{-1}\) at the farthest parameter.
Thus PPMD is more accurate throughout the extrapolation interval, although
its error also increases with the distance from the training set.

\subsection{Error decomposition study}
\label{subsec:cylinder_oracle_study}

To identify the dominant source of the PPMD extrapolation error, we perform
a reference coordinate substitution study while keeping the basis, residual spectral representation, and lifting estimator fixed. Here and
throughout this subsection, \emph{coordinates} refers to the low-dimensional
vectors in the linear and residual spectral representation spaces, rather
than to the physical parameter \(\mu\). For a test trajectory, define
the reference linear coordinates by
\[
z_{\rm ref}(\mu)
:=
\Psi_r^\top
W_{h,\tau}
\bar{\mathbf u}_{h,\tau}(\mu),
\]
and let \(\phi_{\rm ref}(\mu)\) denote the reference nonlinear coordinates
expressed in the same spectral coordinate system used to train the lifting
map. Both \(z_{\rm ref}(\mu)\) and \(\phi_{\rm ref}(\mu)\) are obtained from
the high-fidelity test trajectories and are used only for diagnostic purposes.

The four normalized reconstructions are
\begin{align}
\widehat{\bar{\mathbf u}}^{\rm PPMD}
&=
\Psi_r\widehat z
+
\widehat{\mathcal K}(\widehat\phi),
&
\widehat{\bar{\mathbf u}}^{z{\rm -orc}}
&=
\Psi_r z_{\rm ref}
+
\widehat{\mathcal K}(\widehat\phi),
\nonumber\\
\widehat{\bar{\mathbf u}}^{\phi{\rm -orc}}
&=
\Psi_r\widehat z
+
\widehat{\mathcal K}(\phi_{\rm ref}),
&
\widehat{\bar{\mathbf u}}^{{\rm both-orc}}
&=
\Psi_r z_{\rm ref}
+
\widehat{\mathcal K}(\phi_{\rm ref}).
\label{eq:cylinder_oracle_reconstructions}
\end{align}
The physical-coordinate reconstructions are obtained by applying
\(\mathcal N^{-1}\).

\begin{figure}[t]
\centering
\includegraphics[width=0.82\textwidth]
{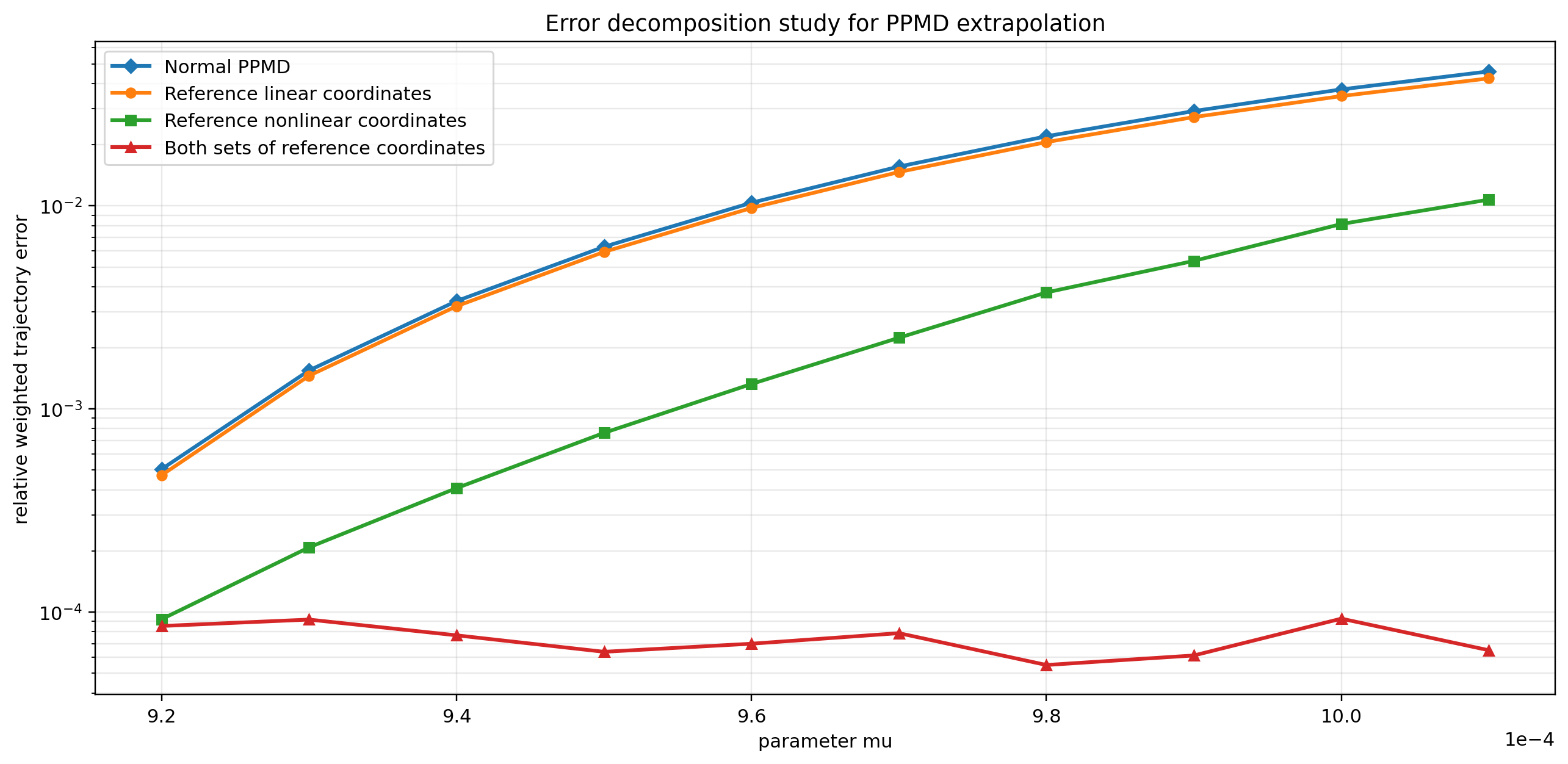}
\caption{
Error decomposition study for the ten extrapolation parameters.
The linear-coordinate substitution replaces only \(\widehat z\), the
nonlinear-coordinate substitution replaces only \(\widehat\phi\), and the
joint substitution replaces both coordinate vectors by their reference
values. All other PPMD components remain fixed.
}
\label{fig:cylinder_oracle_replacement}
\end{figure}

\begin{table}[t]
\centering
\caption{
Mean and maximum trajectory errors over the ten extrapolation parameters.
The final column gives the reduction in mean error relative to standard PPMD.
}
\label{tab:cylinder_oracle_summary}
\begin{tabular}{lccc}
\hline
Reconstruction
& Mean error
& Maximum error
& Reduction
\\
\hline
Standard PPMD
& \(1.7235\times10^{-2}\)
& \(4.5915\times10^{-2}\)
& \(0.00\%\)
\\
Reference linear coordinates
& \(1.6902\times10^{-2}\)
& \(4.4643\times10^{-2}\)
& \(1.93\%\)
\\
Reference nonlinear coordinates
& \(3.302\times10^{-3}\)
& \(1.0730\times10^{-2}\)
& \(80.84\%\)
\\
Both sets of reference coordinates
& \(7.4\times10^{-5}\)
& \(9.3\times10^{-5}\)
& \(99.57\%\)
\\
\hline
\end{tabular}
\end{table}

Replacing only the predicted linear coordinates reduces the mean error by
\(1.93\%\), indicating that the approximation of \(z(\mu)\) is not the
dominant limitation of the standard PPMD prediction. Replacing only the
predicted nonlinear coordinates reduces the mean error by \(80.84\%\),
identifying the approximation of the residual spectral coordinates
\(\phi(\mu)\) as the principal source of extrapolation error in this example.
When both coordinate vectors are replaced by their reference values, the mean
error decreases to \(7.4\times10^{-5}\). The remaining discrepancy includes
the lifting error, the finite-dimensional residual coordinate representation
error, and any error introduced when computing or aligning
\(\phi_{\rm ref}\).

Because the errors in the two coordinate vectors interact through the
reconstruction map, the reported reductions are not additive and should not
be interpreted as an exact decomposition of the total error. Rather, this
diagnostic study measures the sensitivity of the final trajectory
reconstruction to errors in the linear and residual spectral coordinates.
For the present example, the results show that improving the prediction of
the residual spectral coordinates is substantially more important than
further improving the prediction of the linear coordinates or reducing the
lifting error.

\section{Conclusions}
\label{sec:summary}

This work analyzed parametric probabilistic manifold decomposition (PPMD) as a nonintrusive method for approximating complete solution trajectories over a fixed time interval as functions of the model parameters. The method was formulated in a weighted discrete trajectory space whose norm incorporates the spatial discretization and time quadrature. The analysis separates the PDE discretization error from the errors associated with the linear coordinates, residual spectral coordinates, residual lifting, and finite residual representation. Under the stated geometric, spectral,
regression, and stability assumptions, these estimates yield deterministic and conditional high probability reconstruction bounds, together with consistency in probability under joint refinement.

For flow past a cylinder, PPMD produced smaller extrapolation errors than both POD+KRR and POD+AE+KRR under the same data split and latent dimension. The error decomposition study showed that the dominant error arose from predicting the residual spectral coordinates, whereas the linear coordinate and lifting errors were comparatively small. The present results remain conditional on the assumed regularity, spectral approximation, and regression estimates, and the reduced order model is restricted to the prescribed time interval rather than an autonomous time evolution. Future work will focus on more stable prediction
of the residual coordinates, adaptive sampling, higher dimensional parameter domains, and causal extensions of PPMD.

\section{ Acknowledgments}
The authors acknowledge the support of the Top Discipline Plan of Shanghai Universities-Class I and Shanghai Municipal Science and Technology Major Project (No. 2021SHZDZX0100), National Key $R\&D$ Program of China(NO.2022YFE0208000, 2024YFC2816400, and 2024YFC2816401) and the Shanghai Institute of Intelligent Science and Technology, Tongji University.	

\bibliographystyle{siamplain}
\bibliography{references}
\end{document}